\documentclass{gtpart}     
\title{The Universal Covering Space of a Haken $n$--Manifold}

\author{Bell Foozwell}
\givenname{Bell}
\surname{Foozwell}
\address{Trinity College\\ Royal Parade\\ Parkville\\ Vic 3010\\ Australia}
\email{bfoozwel@trinity.unimelb.edu.au}
\urladdr{https://sites.google.com/site/bellfoozwell/}


\arxivreference{}
\arxivpassword{}

\volumenumber{}
\issuenumber{}
\publicationyear{}
\papernumber{}
\startpage{}
\endpage{}
\doi{}
\MR{}
\Zbl{}
\received{}
\revised{}
\accepted{}
\published{}
\publishedonline{}
\proposed{}
\seconded{}
\corresponding{}
\editor{}
\version{}

\newtheorem{theorem}{Theorem}[section]    
\newtheorem{lemma}[theorem]{Lemma}          
\theoremstyle{definition}
\newtheorem{definition}[theorem]{Definition}    
\newtheorem*{question}{Question}
\newcommand{\ul}[1]{\ensuremath{\underline{\underline{#1}}}}
\newcommand{\cov}[1]{\ensuremath{\widetilde{#1}}}
\newcommand{\bd}{\ensuremath{\partial}}

\def \co {\colon\thinspace}
\def \RR {\mathbf{R}}
\def \HH {\mathbf{H}}

\DeclareMathOperator{\inte}{Interior}
\DeclareMathOperator{\closure}{Closure}

\begin{document}

\begin{abstract}
We define the class of Haken $n$--manifolds as a generalisation of Haken 3--manifolds. We prove that the interior of the universal covering of a Haken $n$--manifold is $\RR^n$, which generalises a result of Waldhausen. The techniques used allow us to provide a new proof of Waldhausen's universal cover theorem for Haken 3--manifolds.
\end{abstract}

\maketitle

\section{Introduction}\label{Introduction}
In the final section of Waldhausen's classic paper on Haken 3--manifolds \cite{Wa2}, he proves that the interior of the universal covering space of a Haken 3--manifold is $\RR^3$. A direct generalisation of his proof for Haken $n$--manifolds leads to some difficulties, which are discussed in Foozwell \cite{Fooz}. In this paper, we give a proof of a Waldhausen universal covering theorem in all dimensions, via induction on dimension on the manifold. Surprisingly, the approach in dimension three needs to be different, but we do obtain a new proof of the three-dimensional case that is similar in spirit to the higher dimensional proof.

In section \ref{Preliminary definitions}, we give the basic definitions needed to define Haken $n$--manifolds. The definition is more complicated than the definition for 3--manifolds, and we use the boundary-pattern concept developed by Johannson.

In section \ref{The main result}, we present the proof of the main theorem for higher dimensional Haken manifolds, assuming that the result is true in dimension three.

In section \ref{The three-dimensional case}, we give the proof of the main theorem in the three-dimensional case. We assume the result is true in dimension two, but this can be proved easily using the techniques from section \ref{The main result}, or just using the classical argument.

\section{Preliminary definitions}\label{Preliminary definitions}
\begin{definition}
Let $M$ be a compact $n$--manifold with boundary, and let $\ul m$ be a finite collection of compact, connected $\left(n-1\right)$--manifolds in $\bd M$. Let $i \in \{1, \ldots, n+1\}$. If the intersection of each collection $i$ elements of $\ul m$ is either empty or consists of $\left(n - i\right)$--manifolds, then $\ul m$ is called a {\emph{boundary-pattern}} for $M$.
\end{definition}
Such a manifold is called a {\emph{manifold with boundary-pattern}}. We use the notation $\left(M, \ul m \right)$ when we wish to emphasise that $M$ is a manifold with boundary-pattern. The elements of $ \ul m$ are called \emph{faces} of the boundary-pattern.

We say that a boundary-pattern is {\emph{complete}} if $\bd M = \bigcup \left\{A:A \in \ul m\right\}$. If the boundary-pattern is not complete, a complete boundary-pattern can be formed by including the components of $ \closure \left(\bd M \setminus \bigcup\{A:A \in \ul m \}\right)$ together with the elements of $\ul m$. This complete boundary-pattern is called the {\emph{completion}} of $\ul m$ and is denoted $\overline {\ul m}$.

The {\emph{intersection complex}} $K = K \left(\ul m\right)$ of a manifold with boundary-pattern is
\[
K = \bigcup_{A \in \ul m} \bd A.
\]
The intersection complex of a $3$--manifold with boundary-pattern is a graph with vertices of degree three.

\begin{definition}
Suppose that $\left(M, \ul m\right)$ and $\left(N , \ul n\right)$ are manifolds with boundary-patterns. An {\emph{admissible map}} between $M$ and $N$ is a continuous proper map $f  \co M \rightarrow N$ such that
\[
 \ul m = \bigsqcup_{A \in \ul n} \{B  :     B \mbox{ is a component of } f^{-1}\left(A\right) \}.
\]
Furthermore, $f$ must be transverse to the boundary-patterns.
\end{definition}

Consider a disk with complete boundary-pattern consisting of $i$ components. If $i=1$, then such a disk has zero vertices, and we call such a disk a \emph{zerogon}. For $i \ge 2$, such a disk has $i$ vertices. A \emph{bigon} has two vertices and a \emph{triangle} has three vertices. Collectively, zerogons, bigons and triangles are called \emph{small disks}\footnote{Monogons are disks with one vertex in the boundary. However, these do not play a role here, because they are not manifolds with boundary-pattern. Zerogons were called monogons in Foozwell \cite{Fooz}.}.

\begin{definition}\label{usefulbpttn}
Let $K$ be the intersection complex of an $n$--manifold $\left(M, \ul m\right)$. Suppose that for each admissible map $f  \co \left(\Delta, \ul \delta \right) \rightarrow \left(M, \ul m\right)$ of a small disk, there is a map $g \co \Delta \rightarrow \bd M$, homotopic to $f$ rel $\bd \Delta$, such that
$g^{-1}\left(K\right)$ is the cone on $g^{-1}\left(K\right) \cap \bd \Delta$. Then the boundary-pattern $\ul m$ of $M$ is called a {\emph{useful boundary-pattern}}.
\end{definition}

\begin{definition}\label{esslcurve}
Let $(J, \ul j)$ be a compact one-dimensional manifold with boundary-pattern and let $(M, \ul m)$ be an $n$--dimensional manifold with boundary-pattern. An admissible map $\sigma  \co (J, \ul j) \rightarrow (M, \ul m)$ is an {\emph{inessential}} curve if there is a disk $(\Delta, \ul \delta)$ and an admissible map $g  \co (\Delta, \ul \delta) \rightarrow (M, \ul m)$ such that:
\begin{enumerate}
    \item $J = \closure \left(\bd \Delta \setminus \bigcup \{ A: A \in \ul \delta \} \right)$,
    \item the completion $\left(\Delta, \overline{\ul \delta}\right)$ is a small disk,
    \item $g \vert_J = \sigma$.
\end{enumerate}
Otherwise $\sigma$ is called an {\emph{essential}} curve.
\end{definition}
\begin{definition}\label{essential map}
An admissible map $\varphi  \co (F, \ul f) \rightarrow (M, \ul m)$ is called {\emph{essential}} if for each essential curve $\sigma \co (J, \ul j) \rightarrow (F, \ul f)$ the composition $\varphi \circ \sigma  \co (J, \ul j) \rightarrow (M, \ul m)$ is also essential. In particular, an essential submanifold $F$ of $M$ is a submanifold such that the inclusion map is essential.
\end{definition}

\begin{definition}\label{haken cell}
 A \emph{Haken $1$--cell} is an arc with complete (and useful) boundary-pattern. For $n > 1$, if $(M, \ul m)$ is an $n$--cell with complete and useful boundary-pattern and each face $A \in \ul m$ is a Haken $(n-1)$--cell, then $(M, \ul m)$ is a \emph{Haken $n$--cell}.
\end{definition}
Thus a Haken 1--cell is of the form $([a,b], \{a, b\})$ for $a,b \in \RR$. A Haken 2--cell is a disk with at least four sides.
\begin{definition}\label{induced}
Let $(M, \ul m)$ be a compact $n$--manifold with boundary-pattern. Let $F$ be a codimension-one, properly embedded, two-sided, essential submanifold of $M$.
A boundary-pattern is {\emph{induced}} on $F$ if it is obtained by taking all the intersections of $\partial F$ with the elements of the boundary pattern $\ul m$ on $M$. Equivalently, the boundary-pattern $\ul f$ is the induced pattern if the inclusion $(F, \ul f) \to (M, \ul m)$ is an admissible map.
\end{definition}
\begin{definition}\label{good pair}
Let $(M, \ul m)$ be a compact $n$--manifold with complete and useful boundary-pattern. Let $F$ be a codimension-one, properly embedded, two-sided, essential submanifold of $M$ whose boundary-pattern is induced from the boundary pattern on $M$ and is complete and useful.  Suppose that $F$ is not admissibly boundary-parallel. Then the pair $(M, F)$ is called a {\emph{good  pair}}.
\end{definition}
Suppose that $(M, \ul m)$ is a manifold with boundary-pattern and that $(F, \ul f)$ is a codimension-one submanifold. Let $N$ be the manifold obtained by splitting $M$ open along $F$. There is an obvious map $q \co N \to M$ that glues parts of the boundary of $N$ together to regain $M$. We define a boundary-pattern $\ul n$ on $N$ by
\[
B \in \ul n \text{ if and only if $B$ is a component of } q^{-1}(A)
\]
for $A \in \ul m$ or $A = F$. This is the boundary-pattern that we will use whenever splitting situations arise.

\begin{definition}
A finite sequence
\[ (M_0 , F_0), (M_1, F_1), \dots, (M_k, F_k) \]
of good pairs is called a {\emph{hierarchy of length $k$}} for $M_0$ if the following conditions are satisfied:
\begin{enumerate}
    \item $M_{i+1}$ is obtained by splitting $M_i$ open along $F_i$ and,
    \item $M_{k+1}$ is a finite disjoint union of Haken $n$--cells.
\end{enumerate}
A manifold with a hierarchy is called a {\emph{Haken $n$--manifold}}. A Haken $n$--cell is a Haken manifold with a hierarchy of length zero.
\end{definition}

We regard two Haken $n$--manifolds $M$ and $N$ as equivalent if there is an admissible homeomorphism $\varphi \co \left(M, \ul{m}\right) \to \left(N, \ul{n} \right)$.

\section{The main result}\label{The main result}

In proving our main theorem, we will use the following result of Doyle \cite{Do}.
\begin{theorem}\label{Doyle}
If $P$ is a manifold with interior homeomorphic to $\RR^n$ and boundary homeomorphic to $\RR^{n-1}$, then $P$ is homeomorphic to $\RR^{n-1} \times [0, \infty)$, provided $n \ne 3$.
\end{theorem}
Examples showing that ruling out three-dimensional manifolds is necessary are well known. The paper of Fox and Artin \cite{FA} is a pleasant way to discover such examples.
We will also make use of the following folklore lemma, which in essence is the idea in Stallings \cite{St2}.
\begin{lemma}
Let $P$ be an $n$--manifold that can be written as a countable union of compact subsets. Suppose that for each compact subset $X$ in $P$, there is an embedding $f$ of the standard open ball $B$ into $P$ such that $X \subset f(B)$. Then $P$ is homeomorphic to $\RR^n$.
\end{lemma}
\begin{theorem}
Let $\cov M$ be the universal cover of a Haken $n$--manifold $M$. Then the interior of $\cov M$ is homeomorphic to $\RR^n$.
\end{theorem}

\begin{proof}
We have only defined boundary-patterns for compact manifolds. We will extend the definition to non-compact manifolds, in the case that we have a covering space. If $p \co \cov M \to M$ is a covering map of a manifold with boundary-pattern $(M, \ul m)$, then we define a boundary pattern $p^{-1}(\ul m)$ for $\cov M$ by
\[
p^{-1}(\ul m) = \bigsqcup_{A \in \ul m} \{ B : B \text{ is a component of } p^{-1}(A) \}.
\]

Let the sequence
\[
(M_0, F_0), (M_1,F_1), \dots, (M_k, F_k)
\]
be a hierarchy for $M_0 = M$. To simplify notation let $N = M_1$ and $F = F_0$. We assume that the interior of the universal cover $\cov N$ of $N$ is homeomorphic to $\RR^n$. If $p \co \cov M \to M$ is the covering projection, then the closure of each component of $\cov M \setminus p^{-1}(F)$ is homeomorphic to $\cov N$. There are a countable collection of such pieces and we label them $\{N_1, N_2, N_3, \dots \}$.

Assume that each component of $p^{-1}(F)$ has interior homeomorphic to $\RR^{n-1}$. There are countably many pieces of $p^{-1}(F)$, which we label as $\{F_1, F_2, F_3, \dots \}$.
We arrange the labelling so that $N^1 = N_1$,
\[
N^i \cap N_{i+1} = F_i,
\]
and
\[
N^{i+1} = N^i \cup N_{i+1}.
\]
We first aim to prove that $\inte(N^i) \cong \RR^n$ for each $i \ge 1$. First observe that, by assumption, $\inte(N^1) \cong \RR^n$. We assume that $\inte(N^j) \cong \RR^n$ is true for all $j \le i$ and then prove that $\inte(N^{i+1}) \cong \RR^n$. Recall that $N^{i+1} = N^i \cup N_{i+1}$. Let $P = \inte(N^i) \cup \inte(F_i)$ and let $Q = \inte(F_i) \cup \inte(N_{i+1})$. Both $P$ and $Q$ are manifolds with interior homeomorphic to $\RR^n$ and boundary homeomorphic to $\RR^{n-1}$. By Doyle's theorem \ref{Doyle}, it follows that each of $P$ and $Q$ are homeomorphic to $\RR^{n-1} \times [0,\infty)$, provided we assume that $n > 3$.

Let us regard $P \cup Q$ as being formed by attaching a collar $Q =  \RR^{n-1} \times [0,\infty)$  of the boundary of $P$ to $\bd P$. Thus $P \cup Q \cong \inte(P) \cong \RR^n$. So $\inte(N^{i+1}) \cong \RR^n$ as we aimed to prove.

Let $X$ be a compact subset of $\inte(\cov M)$. Then $X \subset N^i$ for some integer $i$.  Since $\inte(N^i) \cong \RR^n$, it follows that there is an open ball in $\inte(N^i)$ containing $X$. Hence, there is an open ball in $\inte(\cov M)$ containing $X$, which shows that $\inte(\cov M) \cong \RR^n$.
\end{proof}
\section{The three-dimensional case}\label{The three-dimensional case}
We use the following theorem of Doyle and Hocking \cite{DoHo} in this section.
\begin{theorem}\label{Doyle-Hocking}
Let $M$ be a 3--manifold such that $\inte(M) \cong \RR^3$ and $\bd M \cong \RR^2$. Suppose that $M \ne \RR^2 \times [0, \infty)$. Then there is a polygonal graph $\Gamma \subset M$ such that $\Gamma \cap M$ is a point $x$ and there is no closed 3-cell $C$ in $M$ containing $\Gamma$ for which $\Gamma \setminus \{x\} \subset \inte(C)$.
\end{theorem}
We use theorem \ref{Doyle-Hocking} as follows: suppose $Y$ is a manifold with interior homeomorphic to $\RR^3$ and boundary homeomorphic to $\RR^2$, such that every graph $\Gamma \subset Y$ that meets $\bd Y$ in a point can be contained in a ball that meets $\bd Y$ in a disk, then $Y$ is homeomorphic to $\RR^2 \times [0, \infty)$. We will refer to such a ball as an \emph{engulfing ball} for $\Gamma$.

\begin{theorem}
Let $M$ be an orientable Haken 3--manifold. Then the interior of the universal cover $\cov M$ of $M$ is homeomorphic to $\RR^n$.
\end{theorem}
\begin{proof}
Let us first suppose that $M$ is closed. A result of Aitchison and Rubinstein \cite{AR} says that $M$ has a very short hierarchy:
\[
(M,F), (N,S), (P,D),
\]
where $F$ is a maximal collection if closed incompressible surfaces, $S$ is a collection of spanning surfaces, $P$ is a disjoint union of handlebodies and $D$ is a collection of meridian disks in each handlebody.

\begin{lemma} Let $P_i$ be a component of the universal covering space of a component of $P$, and let $E$ be the closed unit ball in $\RR^3$. There is an embedding $e \co P_i \to E$ such that $\inte(E) \subset e(P_i)$, and for each $A \in \cov{\ul{p_i}}$, the closure of $e(A)$ in $\bd E$ is a disk.
\end{lemma}
\begin{proof}
If $P_i$ covers a solid torus, then $P_i$ is homeomorphic to $D^2 \times \RR$, which embeds in the unit ball as $E \setminus \{(0,0,\pm 1)\}$. If $A \in \cov{\ul{p_i}}$, then $e(A)$ is bounded by lines that cover circles in the graph of $\ul p$. Each such line has one end at $(0,0,1)$ and the other at $(0,0,-1)$, so the closure of $e(A)$ in $\bd E$ is a disk.

If $P_i$ covers a handlebody of positive genus, then we can visualise $P_i$ as the regular neighbourhood in $\HH^3$ of a graph in $\HH^2$. Each vertex of the graph has degree four. The graph meets the boundary of $\HH^2$ in a Cantor set. We shall refer to the points in this set as \emph{Cantor points}. The closure of the regular neighbourhood of the graph is a ball. Clearly, there is an embedding $e \co P_i \to E$ into the unit ball. We may choose $e$ so that the Cantor points lie in the equator of $E$. We must show that if $A$ is in the boundary-pattern of $P_i$, then the closure of $e(A)$ in the ball is a disk. Note that if $L$ is a line in the boundary of $e(A)$, then the different ends of $L$ must lie in different Cantor points. Thus, the closure of $e(A)$ is a disk.
\end{proof}
Note that lemma implies that if $T$ is an element of the boundary-pattern of $\cov P$, then $\inte(\cov P) \cup \inte(T)$ is homeomorphic to $\RR^2 \times [0,\infty)$.
\begin{lemma}
Let $\pi \co \cov N \to N$ be the universal covering of a component of $N$, and let $A$ be a component of $\pi^{-1}(\bd N)$. Then $\inte(\cov N) \cup A$ is homeomorphic to $\RR^2 \times [0,\infty)$.
\end{lemma}
\begin{proof}
Let $S_1, S_2, \dots $ denote the components of $\pi^{-1}(S)$. Each $S_i$ is homeomorphic to the universal cover $\cov S$ of $S$. The closure of each component of $\cov N \setminus \pi^{-1}(S)$ is homeomorphic to $\cov P$. We define collections $P_1, P_2, P_3, \dots$ and $P^1, P^2, P^3, \dots$ of submanifolds of $\cov N$ that satisfy the following:
\begin{enumerate}
    \item The collection $\{P_i\}$ covers $\cov N$. That is $\cov N = \bigcup_{i=1}^{\infty} P_i$.
    \item Each $P_i$ is the closure of a component of $\cov N \setminus \pi^{-1}(S)$.
    \item The labelling is arranged so that
    \begin{align*}
        P^1 &= P_1, \\
        P^i \cap P_{i+1} &= S_i, \text{ and }\\
        P_{i+1} &= P^i \cup P_{i+1}.
    \end{align*}
\end{enumerate}
We aim to show that $V := \inte(\cov N) \cup A$ is homeomorphic to $\RR^2 \times [0,\infty)$, where $A$ is a component of $\pi^{-1}(\bd N)$. Let $\Gamma$ be a compact graph in $V$ such that $\Gamma \cap \bd V$ is a point $v$. We will show that there is a ball $B \subset V$ such that $\Gamma \subset B$ and $B \cap \bd V$ is a disk containing $v$.

Since $\Gamma$ is compact, there is some $P^i$ containing $\Gamma$. Therefore, we will show that if $\Gamma \subset P^i$, then there is a ball $B \subset P^i$ such that $\Gamma \subset B$ and $B \cap (A \cap P^i)$ is a disk. We prove this by induction on the index $i$ of the collection $\{P^i\}$. However, we need to prove a stronger statement, and to do this we need to define a boundary-pattern $\ul{p^i}$ inductively for each $P^i$. Since $P^1 = P_1$, we define $\ul{p^1} = \ul{p_1}$. For $i >1$ suppose we have $A_1 \in \ul{p^i}, A_2 \in \ul{p_{i+1}}$ and that $A_1,A_2 \subset \bd P^i$. If $A_1 \cap A_2$ is an arc, the $A_1 \cap A_2 \in \ul{p^{i+1}}$. If $A_1 \cap A_2$ is not an arc, then $A_1$ and $A_2$ belong to distinct elements of $\ul{p^{i+1}}$.

We need to prove that if $\Gamma$ is a compact graph in $P^i$ that has non-empty intersection with finitely many faces of $\ul{p^i}$, then there is a compact ball $B \subset P^i$ containing $\Gamma$ and satisfying the following \emph{face intersection conditions:}
\begin{itemize}
    \item for each $A \in \ul{p^i}$, if $\Gamma \cap A \ne \varnothing$, then $B \cap A$ is a disk,
    \item for each $A \in \ul{p^i}$, if $\Gamma \cap A = \varnothing$, then $B \cap A = \varnothing$.
\end{itemize}
If $\Gamma \subset P^1$, then there is a ball in $P^1$ containing $\Gamma$  and satisfying the face intersection conditions, because $P^1$ is the universal cover of a handlebody. We assume that the result is true for graphs in $P^i$. Let $\Gamma \subset P^{i+1}$. Then $\Gamma_2 = \Gamma \cap P_{i+1}$ is a graph in the universal cover of a handlebody, so there is a ball $B_2 \subset P_{i+1}$ that contains $\Gamma_2$ and satisfies the face intersection conditions. Let $\Gamma_1 = \Gamma \cap P^i$. Then, by assumption, there is a ball $B_1 \subset P^i$ that contains $\Gamma_1$ and satisfies the face intersection conditions. In particular, both $B_1 \cap S_i$ and $B_2 \cap S_i$ are disks.

If $B_1 \cap B_2$ is a disk, then $B_1 \cup B_2$ is the ball we need. If $B_1 \cap B_2$ is not a disk, then we need to modify at least one of $B_1$  or $B_2$.

Observe that $B_1 \cap B_2$ is a compact subset of $\inte(P^1 \cap P_{i+1}) = \inte(S_i) \cong \RR^2$, so there is a disk in $\inte(S_i)$ containing $B_1 \cap B_2$. Let $U$ be a sufficiently small bicollar of this disk. Then $B_1 \cup (U \cap P^i)$ is a ball and so is $B_2 \cup ( U \cap P_{i+1})$.
Now $B_1 \cup B_2 \cup U$ is a ball in $P^{i+1}$ that contains $\Gamma$ and satisfies the face intersection conditions for $P^{i+1}$.

So we have proved: if $\Gamma$  is a compact graph in $P^i$ that has non-empty intersection with finitely many faces of $\ul{p^i}$, then there is a compact ball $B \subset P^i$ that contains $\Gamma$ and satisfies the face intersection conditions. In particular, if $\Gamma \subset V$ such that $\Gamma \cap \bd V$ is a point, then there is a ball $B \subset V$ containing $\Gamma$ such that $B \cap \bd V$ is a disk. Then theorem \ref{Doyle-Hocking} says that $V$ is homeomorphic to $\RR^2 \times [0,\infty)$.
\end{proof}
To show that $\cov M$ is homeomorphic to $\RR^3$, we repeat the above argument. However, the details are less complicated than above, because if $F$ is a closed surface, then its universal cover is $\RR^2$ (rather than  a missing boundary plane).
\end{proof}

\section{Conclusion}\label{Conclusion}
The universal covering space result establishes Haken $n$--manifolds as a special class of spaces worthy of further study. Mike Davis \cite{Da}, for example, has produced examples of aspherical $4$--manifolds with universal covering spaces not homeomorphic to $\RR^n$. We have shown elsewhere \cite{Fooz}, using a direct generalisation of Waldhausen's proof in \cite{Wa1}, that the word problem is solvable for the fundamental group of a Haken $n$--manifold.

Probably the most important open problem at the moment is the question of topological rigidity for Haken $n$--manifolds.
\begin{question}
If $(M, \ul m)$ and $(N, \ul n)$ are Haken $n$--manifolds, that are admissibly homotopy equivalent, are they homeomorphic?
\end{question}
In particular, answering this question in dimension four would be of great interest. The techniques of Waldhausen \cite{Wa2} do not appear to be directly generalisable to the situation in higher dimensions. It seems that a new approach is required.

\section{Acknowledgements} I wish to thank Hyam Rubinstein and James Coffey for many useful ideas and helpful conversations.


\end{document}